\documentclass[11pt]{article}
\usepackage{amssymb,amsmath,latexsym, verbatim,hyperref,amsthm}
\usepackage{tikz}
\allowdisplaybreaks

\begin{document}

\begin{doublespace}

\newtheorem{thm}{Theorem}[section]
\newtheorem{lemma}[thm]{Lemma}
\newtheorem{cond}[thm]{Condition}
\newtheorem{defn}[thm]{Definition}
\newtheorem{prop}[thm]{Proposition}
\newtheorem{corollary}[thm]{Corollary}
\newtheorem{remark}[thm]{Remark}
\newtheorem{example}[thm]{Example}
\newtheorem{conj}[thm]{Conjecture}
\numberwithin{equation}{section}
\def\ee{\varepsilon}
\def\qed{{\hfill $\Box$ \bigskip}}
\def\NN{{\cal N}}
\def\AA{{\cal A}}
\def\MM{{\cal M}}
\def\BB{{\cal B}}
\def\CC{{\cal C}}
\def\LL{{\cal L}}
\def\DD{{\cal D}}
\def\FF{{\cal F}}
\def\EE{{\cal E}}
\def\QQ{{\cal Q}}
\def\RR{{\mathbb R}}
\def\R{{\mathbb R}}
\def\L{{\bf L}}
\def\K{{\bf K}}
\def\S{{\bf S}}
\def\A{{\bf A}}
\def\E{{\mathbb E}}
\def\F{{\bf F}}
\def\P{{\mathbb P}}
\def\N{{\mathbb N}}
\def\eps{\varepsilon}
\def\wh{\widehat}
\def\wt{\widetilde}
\def\pf{\noindent{\bf Proof.} }
\def\beq{\begin{equation}}
\def\eeq{\end{equation}}
\def\lam{\lambda}
\def\H{\mathcal{H}}
\def\nn{\nonumber}
\def\C{\mathbb{C}}

\newcommand\blfootnote[1]{%
  \begingroup
  \renewcommand\thefootnote{}\footnote{#1}%
  \addtocounter{footnote}{-1}%
  \endgroup
}

\title{\Large \bf  Large-time and small-time behaviors of the spectral heat content for time-changed stable processes}
\author{Kei Kobayashi and Hyunchul Park}

\date{\today }
\maketitle

\begin{abstract}
We study the large-time and small-time asymptotic behaviors of the spectral heat content for time-changed stable processes, where the time change belongs to a large class of inverse subordinators. 
For the large-time behavior, the spectral heat content decays polynomially with the decay rate determined by the Laplace exponent of the underlying subordinator, which is in sharp contrast to the exponential decay observed in the case when the time change is a subordinator. On the other hand, the small-time behavior exhibits three different decay regimes, where the decay rate is determined by both the Laplace exponent and the index of the stable process.
\end{abstract}

\section{Introduction}\label{section:Introduction}
The spectral heat content measures the total heat that remains on a domain $\Omega\subset \R^d$ whose initial temperature is one with Dirichlet boundary condition. 
The spectral heat content for Brownian motions has been studied intensively in the past few decades. Recently, there have been growing interests in studying the spectral heat content for jump processes, which is defined by replacing the Brownian motions with other jump processes and putting the zero exterior condition outside $\Omega$ in order to take into account the fact that jump processes could exit the domain by jumping into $\R^{d}\setminus \overline{\Omega}$. In \cite{Val2016, Val2017, GPS19, KP, P20, PS19, PS22}, the spectral heat content for stable processes and other L\'evy processes are studied for both subordinate killed processes and killed subordinate processes. 
In particular, in \cite{PS22}, the two-term asymptotic expansion for the spectral heat content for isotropic stable processes on bounded $C^{1,1}$ open sets was investigated. 

The main purpose of this paper is to study the large-time and small-time asymptotic behaviors of the spectral heat content for time-changed stable processes of the form $Y\circ E=\{Y_{E_t}\}_{t\ge0}$, where the outer process $Y=\{Y_t\}_{t\ge0}$ is an isotropic stable process and the inner process (or the time change) $E=\{E_t\}_{t\ge 0}$ belongs to a large class of \textit{inverse subordinators}. 
In general, the large-time and small-time asymptotics for the spectral heat content provide both spectral information of the underlying process and geometric information of the domain. Indeed, for Brownian motions, the large-time decay rate of the spectral heat content is determined by the principal eigenvalue of the infinitesimal generator of the Brownian motions. 
On the other hand, as for the small-time behavior, the spectral heat content for Brownian motions on a smooth domain admits an asymptotic expansion, with the coefficients of the expansion containing geometric characteristics of the domain $\Omega$, such as the area, perimeter and mean curvature.

The main results of this paper are Theorems \ref{thm:SHC large time} and \ref{thm:SHC small time} below, which provide the large-time and small-time asymptotic behaviors of the spectral heat content $Q_{\Omega}^{Y\circ E}(t)$ for time-changed isotropic stable processes of the form $Y\circ E$.  
Note that throughout the paper, the symbol $f(t)\sim g(t)$ means $f(t)/g(t)\to 1$ as $t\to \infty$ or $t\downarrow 0$, depending on which asymptotic behavior is being considered. 

\begin{thm}\label{thm:SHC large time}
Let $Y$ be an isotropic stable process and $\{(\lam_{n}^{Y}, \psi_{n}^{Y})\}_{n=1}^\infty$ be the eigenpairs of the infinitesimal generator for the associated killed process. Let $E$ be the inverse of a subordinator whose Laplace exponent $\phi$ is regularly varying at $0^+$ with index $\beta\in[0,1)$. Suppose $Y$ and $E$ are independent. 
Let $\Omega\subset \R^{d}$ be a bounded open set.
Then 
\[
Q_{\Omega}^{Y\circ E}(t)
\sim \phi(1/t)\sum_{n=1}^{\infty}\frac{\left(\int_{\Omega}\psi^{Y}_{n}(x)dx\right)^{2}}{\lam^{Y}_{n}\Gamma(1-\beta)} \ \textrm{as} \ t\to\infty.
\]
\end{thm}
\begin{thm}\label{thm:SHC small time}
Let $Y$ be an isotropic stable process of index $\alpha\in (0,2)$. Let $E$ be the inverse of a subordinator whose Laplace exponent $\phi$ is regularly varying at $\infty$ with index $\beta\in(0,1)$. Suppose $Y$ and $E$ are independent. 
Let $\Omega\subset \R^{d}$ be a bounded $C^{1,1}$ open set if $d\geq 2$ or a bounded open interval if $d=1$.
Then as $t\downarrow 0$,
\begin{align*}
|\Omega|-Q^{Y\circ E}_{\Omega}(t)\sim&
\begin{cases}
\dfrac{|\partial \Omega|\E[\overline{Z}_{1}^{(\alpha)}]\Gamma(1+1/\alpha)}{\Gamma(1+\beta/\alpha)}[\phi(1/t)]^{-1/\alpha} &\text{if } \alpha\in (1,2),\vspace{1mm}\\
\dfrac{|\partial \Omega|}{\pi \Gamma(1+\beta)}[\phi(1/t)]^{-1}\ln\phi(1/t)&\text{if }\alpha=1,\vspace{1mm}\\
\dfrac{\mathrm{Per}_{\alpha}(\Omega)}{\Gamma(1+\beta)}[\phi(1/t)]^{-1} &\text{if } 
\alpha\in (0,1).
\end{cases}
\end{align*}
\end{thm}
\noindent
Detailed information about the notations used in the theorems, including $\overline{Z}_{1}^{(\alpha)}$ and $\mathrm{Per}_{\alpha}(\Omega)$, appear in the preliminary section, Section \ref{section:Preliminaries}.

Theorem \ref{thm:SHC large time} is established in Section \ref{subsection:Large Time Behavior}.
To the authors' knowledge, this is the very first paper in the literature that establishes the \textit{large-time} behavior of the spectral heat content for time-changed processes with the time changes being \textit{inverse} subordinators. Moreover, the result turns out to be very different from the case when the time changes are subordinators themselves. Indeed, as Proposition \ref{prop:SHC S} shows, if the time-change is given by a subordinator, then the spectral heat content exhibits an exponential decay and the decay rate is determined by $\phi(\lam_{1}^{Y})$, where $\phi$ is the Laplace exponent of the subordinator and $\lam^{Y}_{1}$ is the principal eigenvalue of the infinitesimal generator for the associated killed process. 
In contrast, in the case of an \textit{inverse} subordinator, Theorem \ref{thm:SHC large time} indicates that the spectral heat content exhibits a polynomial decay and the information about all the eigenpairs $\{(\lam_{n}^{Y}, \psi_{n}^{Y})\}_{n=1}^{\infty}$ appear in the limit.
In fact, the large-time behavior of the spectral heat content can be derived in a much more general setting; one can replace the isotropic stable process $Y$ with any L\'evy process
 whose associated killed process has transition density with representation of the form \eqref{eqn:HK} (see Remark \ref{remark:large}, item 4) for details).

Section \ref{subsection:Small Time Behavior} is devoted to the derivation of Theorem \ref{thm:SHC small time}.   
We assume that the domain $\Omega$ is a bounded $C^{1,1}$ open set if $d\geq 2$ or a bounded open interval if $d=1$ so that we can use a recent result \cite[Theorem 1.1]{PS22} or \cite[Theorem 1.1]{Val2016} on the spectral heat content for isotropic stable processes. 

The short-time behavior of the spectral heat content for time-changed Brownian motions, where the time change is given by an inverse subordinator (as opposed to a subordinator), was first studied in \cite{KP}. This current paper can be regarded as a natural continuation of the investigation carried out in \cite{KP} since Brownian motions are stable processes with  index $\alpha=2$ (and they are the only stable processes with continuous sample paths). However, let us stress again that Theorem \ref{thm:SHC large time} on the large-time behavior is completely new, and our approach for proving the theorem is significantly different from those in \cite{KP}; in particular, we resort to the double Laplace transform of the inverse subordinator $E$ (see Lemma \ref{lemma:LT}). Moreover, even for the small-time behavior, replacing Brownian motions with stable processes has generated a new difficulty that was not present in \cite{KP}.
Indeed, when the stability index is $\alpha=1$, even though the exact asymptotic behavior of the function $t\mapsto \E[E_{t}\ln(1/E_{t})]$ is needed, a standard argument based on the Tauberian theorem together with the monotone density theorem is not directly applicable since the function $x\mapsto x\ln(1/x)$ is not monotone.
We will overcome this difficulty by introducing a \textit{monotonized} function $V(x)$ of $x\ln(1/x)$ (see \eqref{eqn:F} for the definition) and showing that the error induced by the introduction of $V(x)$ is negligible in Proposition \ref{prop:asym log}.

Before closing this section, let us briefly explain why it is worth investigating processes involving inverse subordinators. First, time-changed Brownian motions with time changes being inverse subordinators naturally appear in the context of subdiffusions, where the particles diffuse at a slower pace than the usual Brownian particles. Moreover, over the past few decades, the time-changed Brownian motions and their various generalizations, including time-changed L\'evy processes, time-changed fractional Brownian motions, and solutions of stochastic differential equations driven by such time-changed processes, have been widely studied due to a number of practical applications arising in physics, biology, hydrology, finance, etc.\ (see e.g.\ \cite{MeerschaertSikorskii,HKU-book} and references therein). One of the focuses of these investigations is to identify the governing equations for the time-changed processes. In particular, it is known that, when the outer process $Y$ is a L\'evy process and the independent time change $E$ is an inverse \textit{stable} subordinator with index $\beta\in(0,1)$, then the governing equation for the time-changed L\'evy process $Y\circ E$ is the partial differential equation $\partial^\beta_t u(t,x)=A_x u(t,x)$, with $A_x$ being the infinitesimal generator of $Y$ acting on $x$ and $\partial^\beta_t$ denoting the Caputo fractional derivative of order $\beta$ acting on $t$. 
Theorems \ref{thm:SHC large time} and \ref{thm:SHC small time} to be established in this paper are valuable since they provide more information about the nature of such time-changed L\'evy processes.

\section{Preliminaries}\label{section:Preliminaries}
Let $Y=\{Y_{t}\}_{t\geq 0}$ be an \textit{isotropic stable process} of stability index $\alpha\in(0,2]$ with c\`adl\`ag paths, the characteristic function of which is given by $\E[e^{i\xi Y_{t}}]=e^{-t|\xi|^{\alpha}}$, $\xi\in\R^d$.
When $\alpha=2$, $Y$ is a Brownian motion whose sample paths are continuous, whereas for $\alpha\in (0,2)$, $Y$ is a pure-jump process. 

Let $\Omega$ be a bounded open set in $\R^{d}$. 
The \textit{spectral heat content} $Q_{\Omega}^{Y}(t)$ for the stable process $Y$ on $\Omega$ at time $t$ is defined by 
\begin{align}\label{def:spectralheatcontent}
Q_{\Omega}^{Y}(t)=\int_{\Omega}\P_{x}(\tau_{\Omega}^{Y}>t)dx, 
\end{align}
where $\tau^{Y}_{\Omega}=\inf\{t>0 : Y_{t}\notin \Omega\}$ is the first exit time of $Y$ from $\Omega$.
Let $p^{Y}(t,x,y)$ be the transition density for $Y$. By Fourier inversion, 
$
p^{Y}(t,x,y)=p^{Y}(t,y-x)=(2\pi)^{-d}\int_{\R^{d}}e^{-i\langle \xi, y-x\rangle}e^{-t|\xi|^{\alpha}}d\xi.
$
In particular, $p^{Y}(t,x,y)\leq (2\pi)^{-d}\int_{\R^{d}}e^{-t|\xi|^{\alpha}}d\xi=\frac{\omega_{d}}{\alpha}\Gamma(d/\alpha)t^{-d/\alpha}$, where $\omega_{d}=\frac{2\pi^{d/2}}{\Gamma(d/2)}$.
Let $Y^{\Omega}=\{Y^{\Omega}_{t}\}_{t\geq 0}$ be the killed process defined by $Y^{\Omega}_{t}=Y_{t}$ if $t<\tau^{Y}_{\Omega}$ and $Y^{\Omega}_{t}=\partial$ if $t\geq \tau^{Y}_{\Omega}$,
where $\partial$ is a cemetery state.
The transition density $p^{Y,\Omega}(t,x,y)$ for $Y^{\Omega}$ is given by 
$
	p^{Y,\Omega}(t,x,y)=p^{Y}(t,x,y)-r^{Y}_{\Omega}(t,x,y),
$
 where $r^{Y}_{\Omega}(t,x,y)=\E[p^{Y}(t-\tau^{Y}_{\Omega}, Y_{\tau_{\Omega}^{Y}},y); \tau^{Y}_{\Omega} \leq t]$. 
In particular, $p^{Y,\Omega}(t,x,y)\leq p^{Y}(t,x,y)\leq \frac{\omega_{d}}{\alpha}\Gamma(d/\alpha)t^{-d/\alpha}.$ 
 Hence, the semigroup defined by $T_{t}^{Y,\Omega}f(x):=\E_{x}[f(Y_{t}^{\Omega})]$ for $f\in L^{2}(\Omega)$ is a Hilbert--Schmidt operator, and there exist pairs $\{(\lam_{n}^{Y}, \psi_{n}^{Y})\}_{n=1}^\infty$ of eigenvalues and eigenfunctions (or eigenpairs in short) such that 
 \begin{align}\label{eqn:eigenvalues}
	0< \lam_1^Y < \lam_2^Y \le \cdots \le \lam_n^Y \to \infty
 \end{align}
 and
\beq\label{eqn:HK}
p^{Y,\Omega}(t,x,y)=\sum_{n=1}^{\infty}e^{-\lam^{Y}_{n}t}\psi^{Y}_{n}(x)\psi^{Y}_{n}(y).
\eeq
Due to the identity $\P_{x}(\tau^{Y}_{\Omega}>t)=\int_{\Omega}p^{Y,\Omega}(t,x,y)dy$, the spectral heat content for $Y$ has the alternative representation
\beq\label{eqn:SH2}
Q^{Y}_{\Omega}(t)
=\sum_{n=1}^{\infty}e^{-\lam^{Y}_{n}t}\left(\int_{\Omega}\psi^{Y}_{n}(x)dx\right)^{2}.
\eeq

Now we state a recent result from \cite{Val2016, PS22} about the two-term small-time asymptotic behavior of $Q_{\Omega}^{Y}(t)$, where the open set $\Omega\subset \R^d$ when $d\ge 2$ is assumed to be $C^{1,1}$; i.e., its boundary can be locally represented as the graph of a $C^{1}$ function whose gradient is Lipschitz.
Define
\begin{align}\label{def:f_alpha}
f_{\alpha}(t)=
\begin{cases}
t^{1/\alpha} &\text{if } \alpha\in (1,2),\\
t\ln(1/t) &\text{if } \alpha=1,\\
t &\text{if }\alpha\in(0,1),
\end{cases}
\ \ \ \textrm{and} \ \ \ 
c_{\alpha}=
\begin{cases}
\E[\overline{Z}_{1}^{(\alpha)}]|\partial \Omega| &\text{if } \alpha\in (1,2),\\
\dfrac{|\partial \Omega|}{\pi} &\text{if } \alpha=1,\\
\mathrm{Per}_{\alpha}(\Omega)&\text{if }\alpha\in(0,1).
\end{cases}
\end{align}
Here, $|\partial \Omega|$ is the perimeter of $\Omega$ if $d\geq 2$ or $|\partial \Omega|$=2 if $d=1$ (in which case, $\Omega$ is a bounded open interval),
$\overline{Z}_{t}^{(\alpha)}=\sup_{s\le t}Z^{(\alpha)}_s$ stands for the running supremum of a one-dimensional symmetric $\alpha$-stable process $\{Z^{(\alpha)}_t\}_{t\ge 0}$, and $\text{Per}_{\alpha}(\Omega)=\int_{\Omega}\int_{\Omega^{c}}\frac{c(d,\alpha)}{|x-y|^{d+\alpha}}dydx$ is the $\alpha$-fractional perimeter of $\Omega$. 
\begin{thm}[{\cite[Theorem 1.1]{Val2016}} and {\cite[Theorem 1.1]{PS22}}]\label{thm:main}
Let $\Omega\subset \R^d$ be a bounded $C^{1,1}$ open set if $d\geq 2$ or a bounded open interval if $d=1$.
Let $Y$ be an isotropic stable process of index $\alpha\in(0,2)$. 
Let $f_{\alpha}(t)$ and $c_{\alpha}$ be defined as in \eqref{def:f_alpha}. Then 
\[
	|\Omega|-Q^{Y}_{\Omega}(t)\sim c_{\alpha}f_{\alpha}(t) \ \textrm{as} \ t\downarrow 0.
\]
\end{thm}

The main object of study in this paper is the spectral heat content for time-changed isotropic stable processes. More precisely, let $Y$ be the isotropic stable process discussed above, and 
let $D=\{D_{t}\}_{t\geq0}$ be an independent \textit{subordinator} (one-dimensional L\'evy process with nondecreasing c\`adl\`ag paths starting from 0).
Let $\phi$ denote the Laplace exponent of $D$, so that 
\begin{align}\label{def-laplaceexponent}
\E[e^{-\lam D_{t}}]=e^{-t\phi(\lam)}=\exp\left(-t\int_{0}^{\infty}(1-e^{-\lam x})\nu(dx)\right), \quad \lam,t>0,
\end{align}
with the L\'evy measure $\nu$ satisfying $\int_{0}^{\infty}(1\wedge x)\nu(dx)<\infty$. 
Throughout the paper, we assume that the L\'evy measure is infinite; i.e.,\ $\nu(0,\infty)=\infty$, which is equivalent to saying that $\phi(\lam)\to \infty$ as $\lam\to\infty$. 
The Laplace exponent $\phi$ is strictly increasing since $\phi(\lam_2)-\phi(\lam_1)=\int_0^\infty (e^{-\lam_1 x}-e^{-\lam_2 x})\nu(dx)>0$ whenever $0<\lam_1<\lam_2$. Moreover, $\phi(0^+)=0$, and $\phi$ is a Bernstein function; i.e.,\ $(-1)^n\phi^{(n)}\leq 0$ for all integers $n\ge 1$.
Now, let $E=\{E_{t}\}_{t\geq 0}$ be the inverse of $D$ defined by 
\[
E_{t}=\inf\{u>0 : D_{u}>t\}, \quad t>0.
\]
The condition that the L\'evy measure $\nu$ is infinite implies that $D$ has strictly increasing paths with jump times being dense in $(0,\infty)$ (see \cite[Theorem 21.3]{Sato}). This in turn implies that the sample paths of the inverse $E$ are continuous, nondecreasing, and starting from 0. 
The spectral heat content $Q_{\Omega}^{Y\circ E}(t)$ for the time-changed isotropic stable process $Y\circ E$ is defined by
\begin{align}\label{def:QYE}
Q_{\Omega}^{Y\circ E}(t)=\E[Q^{Y}_{\Omega}(E_t)]=\int_{\Omega}\P_{x}(\tau_{\Omega}^{Y}>E_{t})dx.
\end{align}

We assume that the Laplace exponent $\phi$ of the subordinator $D$ is \textit{regularly varying} with index $\beta$ at $\infty$ or at 0; i.e.,\ for each $a>0$, 
$\frac{\phi(a\lam)}{\phi(\lam)}\to a^{\beta}$ as $\lam\to\infty$ or as $\lam\downarrow 0$.
Let $\mathcal{R}_{\beta}(\infty)$ and $\mathcal{R}_{\beta}(0^+)$ denote the classes of regularly varying functions with index $\beta$ at $\infty$ and at 0, respectively. 
Any function $f\in\mathcal{R}_{\beta}(\infty)$ or $f\in\mathcal{R}_{\beta}(0^+)$ can be rewritten as
$
f(\lam)=\lam^{\beta}\ell(\lam)
$
with some slowly varying function $\ell\in\mathcal{R}_{0}(\infty)$ or $\ell\in\mathcal{R}_{0}(0^+)$ (see \cite[Theorem 1.4.1]{BGT}). By the Tauberian theorem, the condition that $\phi\in \mathcal{R}_{\beta}(\infty)$ or $\phi\in \mathcal{R}_{\beta}(0^+)$ determines the behavior of the subordinator $D$ at 0 or $\infty$.
Examples of Laplace exponents include $\phi(\lam)=\lam^\beta$ with $\beta\in(0,1)$, which belongs to $\mathcal{R}_{\beta}(\infty)\cap \mathcal{R}_{\beta}(0^+)$ and corresponds to a $\beta$-stable subordinator, and $\phi(\lam)=(\lam+\kappa)^\beta-\kappa^\beta$ with $\beta\in(0,1)$ and $\kappa>0$, which belongs to $\mathcal{R}_{\beta}(\infty)\cap \mathcal{R}_{1}(0^+)$ and corresponds to a tempered stable subordinator.

\section{The Large-Time Behavior}\label{subsection:Large Time Behavior}
This section studies the large-time asymptotic behavior of the spectral heat content for time-changed stable processes when the time change is given by either a subordinator $D$ with infinite L\'evy measure or its inverse $E$. In either case, $\phi$ denotes the Laplace exponent of the underlying subordinator $D$. 
The main assumption of this section is:
\[
\phi\in \mathcal{R}_{\beta}(0^+) \textrm{ with } \beta\in [0,1) \text{ and } \Omega \text{ is a bounded open set}.
\]

We first discuss the case of an inverse subordinator $E$. Our argument relies on two lemmas. The first lemma concerns the exact expression of the double Laplace transform of $E$, while the second concerns the spectral heat content $Q_{\Omega}^Y(t)$ for the stable process $Y$ at $t=0$.
For the remainder of the paper, let the expression $\LL_t[f(t)](s)$ denote the Laplace transform of a function $f$ whenever it exists; i.e.,\ $\LL_t[f(t)](s)=\int_0^\infty f(t)e^{-st}dt$. 

\begin{lemma}[{\cite[Lemma 2]{Kei}}]\label{lemma:LT}
Let $E$ be the inverse of a subordinator $D$ with Laplace exponent $\phi$. Then for any fixed $a>0$, the Laplace transform of the function $t\mapsto \E[e^{-aE_{t}}]$ exists and is given by
\[
\LL_{t}\bigl[\E[e^{-aE_{t}}]\bigr](s)=\frac{\phi(s)}{s}\frac{1}{\phi(s)+a}, \quad s>0.
\]
\end{lemma}

\begin{lemma}\label{lemma:expression at 0}
Let $\Omega$ be a bounded open set.  
Let $Y$ be an isotropic stable process and $\{(\lam_{n}^{Y}, \psi_{n}^{Y})\}_{n=1}^\infty$ be the eigenpairs of the infinitesimal generator for the associated killed process. 
 Then the map $t\mapsto Q_{\Omega}^Y(t)$ is right-continuous at 0. Furthermore,  
\[
Q_{\Omega}^Y(0^+)=|\Omega|=\sum_{n=1}^{\infty}\left(\int_{\Omega}\psi^Y_{n}(x)dx\right)^{2}.
\]
\end{lemma}

\begin{proof}
By the right-continuity of the sample paths of $Y$, $\lim_{t\downarrow 0}\P_{x}(\tau_{\Omega}^Y>t)=\P_{x}(\tau_{\Omega}^Y>0)=1$ for all $x\in \Omega$. Since $|\Omega|<\infty$, 
it follows from representation \eqref{def:spectralheatcontent} and the dominated convergence theorem that
$\lim_{t\downarrow 0}Q_{\Omega}^Y(t)=|\Omega|.$ 
On the other hand, since the function $t\mapsto e^{-\lam_{n}^Yt}$ increases as $t$ decreases for any fixed $n$, the monotone convergence theorem applied to the alternative representation of $Q_{\Omega}^Y(t)$ in \eqref{eqn:SH2} yields
$\lim_{t\downarrow 0}Q_{\Omega}^Y(t)=\sum_{n=1}^{\infty}\left(\int_{\Omega}\psi_{n}^Y(x)dx\right)^{2}.$ This completes the proof.
\end{proof}

\begin{proof}[Proof of Theorem \ref{thm:SHC large time}]
By \eqref{eqn:SH2} and \eqref{def:QYE},
\beq\label{eqn:SHC IS}
Q_{\Omega}^{Y\circ E}(t)=\E[Q_{\Omega}^Y(E_{t})]=\sum_{n=1}^{\infty}E[e^{-\lam^{Y}_{n}E_{t}}]\left(\int_{\Omega}\psi^{Y}_{n}(x)dx \right)^{2},
\eeq
where the expectation and summation signs are interchangeable since the integrand is nonnegative. 
Express $\phi\in \mathcal{R}_\beta(0^+)$ as $\phi(s)=s^{\beta}\ell(s)$ using some $\ell \in \mathcal{R}_0(0^+)$.
 By Lemma \ref{lemma:LT} and the fact that $\phi(0^+)=0$,
\beq\label{eqn:double LT}
\mathcal{L}_{t}\bigl[\E[e^{-\lam^{Y}_{n}E_{t}}]\bigr](s)=\frac{\phi(s)}{s}\frac{1}{\phi(s)+\lam^{Y}_{n}}\sim \frac{\ell(s)}{s^{1-\beta}}\frac{1}{\lam^{Y}_{n}} \  \textrm{as} \ s\downarrow 0.
\eeq
It follows from Karamata's Tauberian Theorem \cite[Theorem 1.7.1]{BGT} that 
$
\int_{0}^{t}\E[e^{-\lam^{Y}_{n}E_{u}}]du \sim \frac{t^{1-\beta}\ell(1/t)}{\lam^{Y}_{n}\Gamma(2-\beta)}
$
as $t\to\infty$,
and hence, by the monotone density theorem (\cite[Theorem 1.7.2]{BGT}), 
\beq\label{eqn:E asymp}
\E[e^{-\lam^{Y}_{n}E_{t}}]\sim \frac{t^{-\beta}\ell(1/t)}{\lam^{Y}_{n}\Gamma(1-\beta)}
=\frac{\phi(1/t)}{\lam^{Y}_{n}\Gamma(1-\beta)}
 \text{ as } t\to\infty.
\eeq
In particular, for $n=1$, there exists $M>0$ such that $\bigl|[\phi(1/t)]^{-1}\E[e^{-\lam^{Y}_{1}E_{t}}]-\frac{1}{\lam^{Y}_{1}\Gamma(1-\beta)}\bigr| <1$ for all $t>M$.
Then it follows from \eqref{eqn:eigenvalues} that for all $t>M$ and $n\ge 1$,
\[
	[\phi(1/t)]^{-1}\E[e^{-\lam^{Y}_{n}E_{t}}]\left(\int_{\Omega}\psi^{Y}_{n}(x)dx\right)^{2}<\left(1+\frac{1}{\lam^{Y}_{1}\Gamma(1-\beta)}\right)\left(\int_{\Omega}\psi^{Y}_{n}(x)dx\right)^{2}.
\]
The latter is summable due to the assumption that $|\Omega|<\infty$ and Lemma \ref{lemma:expression at 0}; therefore, the dominated convergence theorem together with \eqref{eqn:SHC IS} and \eqref{eqn:E asymp} yields
$
\lim_{t\to\infty}[\phi(1/t)]^{-1}Q_{\Omega}^{Y\circ E}(t)
=\sum_{n=1}^{\infty}\frac{(\int_{\Omega}\psi^{Y}_{n}(x)dx)^{2}}{\lam^{Y}_{n}\Gamma(1-\beta)},
$
as desired.
\end{proof}

We now turn our attention to the case when the time-change is a \textit{subordinator} $D$ independent of the isotropic stable process $Y$. 
The spectral heat content for the time-changed stable process $Y\circ D$ on a bounded domain $\Omega$ is defined by 
\[
	Q^{Y\circ D}_{\Omega}(t)=\E[Q^{Y}_{\Omega}(D_{t})].
\]
Note that $Q^{Y\circ D}_{\Omega}(t)$ corresponds to the spectral heat content for \textit{subordinate killed stable processes}.
The large-time behavior of this quantity is given by the following proposition.

\begin{prop}\label{prop:SHC S}
Let $Y$ be an isotropic stable process and $\{(\lam_{n}^{Y}, \psi_{n}^{Y})\}_{n=1}^\infty$ be the eigenpairs of the infinitesimal generator for the associated killed process. Let $D$ be a subordinator with Laplace exponent $\phi$. Suppose $Y$ and $D$ are independent. 
Let $\Omega\subset \R^{d}$ be a bounded open set in $\R^{d}$.
Then 
\[
	\ln Q^{Y\circ D}_{\Omega}(t)\sim -t\phi(\lam_{1}^{Y}) \ \textrm{as} \ t\to\infty.
\]
\end{prop}
\begin{proof}
By representations \eqref{eqn:SH2} and \eqref{def-laplaceexponent}, the spectral heat content $Q^{Y\circ D}_{\Omega}(t)$ can be re-expressed as  
\[
Q^{Y\circ D}_{\Omega}(t)=\sum_{n=1}^{\infty}\E[e^{-\lam^Y_{n}D_{t}}]\left(\int_{\Omega}\psi^{Y}_{n}(x)dx\right)^{2}
=\sum_{n=1}^{\infty}e^{-t\phi(\lam^{Y}_{n})}\left(\int_{\Omega}\psi^{Y}_{n}(x)dx\right)^{2},
\]
where the expectation and summation signs are interchangeable since the integrand is nonnegative. 
Since the Laplace exponent $\phi$ is strictly increasing and the eigenvalues $\{\lam^Y_n\}$ satisfy \eqref{eqn:eigenvalues}, it follows that $0<\phi(\lam^{Y}_1)<\phi(\lam^{Y}_2)\leq\cdots$. This together with the above representation of $Q^{Y\circ D}_{\Omega}(t)$ yields the desired conclusion. 
\end{proof}

\begin{remark}\label{remark:large}
\begin{em}
1) Theorem \ref{thm:SHC large time} shows that $Q_{\Omega}^{Y\circ E}(t)$ exhibits a polynomial decay as $t\to\infty$, which is in sharp contrast to an exponential decay for $Q_{\Omega}^{Y\circ D}(t)$ indicated by Proposition \ref{prop:SHC S}.
Moreover, 
the information about all the eigenpairs $\{(\lam^{Y}_{n},\psi^{Y}_{n})\}_{n=1}^{\infty}$ appears in the limiting expression for $Q_{\Omega}^{Y\circ E}(t)$, whereas only $\lam_{1}^{Y}$ plays a major role in the large-time behavior of $Q_{\Omega}^{Y\circ D}(t)$. 
The difference between the decay rates of $Q_{\Omega}^{Y\circ E}(t)$ and $Q_{\Omega}^{Y\circ D}(t)$ can be ascribed to the fact the introduction of the inverse subordinator $E$ as a time change makes the heat particles diffuse at a slower pace than those with the subordinator $D$ incorporated as a time change. 

2) The decay rate of $Q_{\Omega}^{Y\circ E}(t)$ is determined by the Laplace exponent $\phi$ of $D$ through the condition $\phi\in\mathcal{R}_{\beta}(0^+)$, $\beta\in [0,1)$.
For example, if $\phi(\lam)$ is given by $\phi(\lam)=\lam^{a}+\lam^{b}$ with $0\leq a<b\leq 1$, which implies that the time change $E$ is given by the inverse of the sum of independent stable subordinators with different indices $a$ and $b$, then since $\phi(\lam)\sim \lam^{a}$ as $\lam\downarrow 0$, the large-time asymptotic behavior of $Q_{\Omega}^{W\circ E}(t)$ is given by a 
constant multiple of $t^{-a}$.

3) Theorem \ref{thm:SHC large time} does not include the case when $\beta=1$ since the argument given in \eqref{eqn:SHC IS}--\eqref{eqn:double LT} would fail if $\beta=1$. We believe that the case when $\beta=1$ gives an exponential decay rather than the polynomial decay. 
In fact, if the time change is given by $E_t = t$, in which case $\phi(\lam)=\lam \in \mathcal{R}_1(0^+)$, 
then Proposition \ref{prop:SHC S} implies that 
$\ln Q^{Y}_{\Omega}(t)\sim -t\lam_{1}^{Y}$ as $t\to\infty$.

4) Theorem \ref{thm:SHC large time} and Proposition \ref{prop:SHC S} are stated with the outer process $Y$ taken to be an isotropic stable process; however, similar statements actually hold for much more general outer processes. 
In fact, as long as the killed process associated with $Y$ has transition density with representation of the form \eqref{eqn:HK}, the proofs given to Theorem \ref{thm:SHC large time} and Proposition \ref{prop:SHC S} continue to work.
\end{em}
\end{remark}

\section{The Small-Time Behavior}\label{subsection:Small Time Behavior}
This section is devoted to the analysis of the spectral heat content for time-changed stable processes $Y\circ E$ as $t\downarrow 0$. 
The main assumption of this section is: 
\begin{align*}
&\phi\in \mathcal{R}_{\beta}(\infty) \ \textrm{with} \ \beta\in (0,1) \text{ and }\\
&\Omega \text{ is a bounded } C^{1,1} \text{ open set if }d\geq 2 \text{ or a bounded open interval if }d=1.
\end{align*}
Our argument builds upon \cite[Propositions 2.2 and 4.2]{KP} and the ideas presented in their proofs. 
However, they do not immediately yield Theorem \ref{thm:SHC small time} when the stability index of $Y$ is $\alpha=1$. 
More precisely, the threshold case $\alpha=1$ requires the analysis of the small-time behavior of $\E[E_{t}\ln(1/E_{t})]$,
but as the map $x\mapsto x\ln(1/x)$ is not monotone, 
the method of finding the asymptotic behavior based on 
Karamata's Tauberian theorem and the monotone density theorem is not directly applicable.
To overcome this difficulty, let us introduce the following \textit{monotonized} function $V(x)$ of $x\ln(1/x)$:
\beq\label{eqn:F}
V(x)=x\ln(1/x)\mathbf{1}_{\{0<x\leq e^{-1}\}} + e^{-1}\mathbf{1}_{\{x> e^{-1}\}}, \quad x>0.
\eeq
Clearly, $V(x)$ is nondecreasing on $(0,\infty)$ and it agrees with $x\ln(1/x)$ when $0<x\leq e^{-1}$. 

To establish the small-time behavior of $\E[E_{t}\ln(1/E_{t})]$, we need two lemmas below.
Note that throughout this section, the notation $\E[X; A]$ is often used to represent the expectation $\E[X\mathbf{1}_A]$. 

\begin{lemma}[{\cite[Equations (4.11) and (4.6)]{KP}}]
Let $E$ be the inverse of a subordinator $D$ with Laplace exponent $\phi\in\mathcal{R}_\beta(\infty)$ with $\beta\in(0,1)$. Then for any $p>0$ and $\delta>0$, there exist constants $c_1,c_2>0$ and functions $\ell_1,\ell_2\in \mathcal{R}_{0}(0^+)$ such that
\begin{align}
\label{eqn:small1} &-\ln\P(E_{t}>\delta)=-\ln\P(D_{\delta}<t)\sim c_1t^{-\frac{\beta}{1-\beta}}\ell_1(t) \text{ as } t\downarrow 0;\\
\label{eqn:small2} &-\ln\E[E_{t}^{p}; E_{t}>\delta]\sim c_2 t^{-\frac{\beta}{1-\beta}}\ell_{2}(t) \text{ as } t\downarrow 0.
\end{align}
\end{lemma}

\begin{lemma}\label{lemma:F}
Let $E$ be the inverse of a subordinator $D$ with Laplace exponent $\phi\in\mathcal{R}_\beta(\infty)$ with $\beta\in(0,1)$. 
Let the function $V(x)$ be defined as in \eqref{eqn:F}. Then 
\[
\E[V(E_t)]\sim \frac{1}{\Gamma(1+\beta)}[\phi(1/t)]^{-1}\ln\phi(1/t) \text{ as } t\downarrow 0.
\]
\end{lemma}

\begin{proof}
By \cite[Equation (4.5)]{KP}, for any fixed $u>0$, the Laplace transform of $t\mapsto \P(E_{t}>u)$ is given by 
\beq\label{eqn:aux1}
\mathcal{L}_t[\P(E_{t}>u)](s)
=\int_{0}^{\infty}e^{-st}\P(D_{u}<t)dt=\frac{e^{-u\phi(s)}}{s}.
\eeq
Define 
\[
	g(t)=\E[E_t \ln(E_t); E_t\le e^{-1}].
\]
Then since $\frac{d}{du}(u\ln u)=1+\ln u$,  
\begin{align*}
g(t)
&=\int_{0}^{e^{-1}}(x\ln x)\P(E_t\in dx)
=\int_{0}^{e^{-1}}\biggl(\int_{0}^{x}(1+\ln u)du\biggr)\P(E_t\in dx)\\
&=\int_{0}^{e^{-1}}\biggl((1+\ln u)\int_{u}^{e^{-1}}\P(E_t\in dx)\biggr)du
=\int_{0}^{e^{-1}}(1+\ln u) \P(u<E_t\le e^{-1})du,
\end{align*}
where the Fubini Theorem is applicable since $\ln x+1<0$ for all $x\in(0,e^{-1}]$. 
By the identity $\P(u<E_t\le e^{-1})=\P(E_t>u)-\P(E_t>e^{-1})$ and formula \eqref{eqn:aux1}, the Laplace transform of $g(t)$ is calculated as
$
	\mathcal{L}_t[g(t)](s) 
	=\frac 1s\int_{0}^{e^{-1}}(1+\ln u) \bigl( e^{-u\phi(s)}-e^{-e^{-1}\phi(s)}\bigr)\,du.
$
By integration by parts and change of variables via $v=u\phi(s)$, 
\begin{align*}
	\mathcal{L}_t[g(t)](s) 
	&=\frac 1s\biggl(\biggl[ u\ln u  \Bigl( e^{-u\phi(s)}-e^{-e^{-1}\phi(s)}\Bigr)\biggr]_0^{e^{-1}} +\phi(s) \int_0^{e^{-1}} (u\ln u) e^{-u\phi(s)}\,du\biggr)\\
	&=\frac {1}{s}  \int_0^{e^{-1}\phi(s)} \frac{v}{\phi(s)}\ln \biggl(\frac{v}{\phi(s)}\biggr) e^{-v}\,dv\\
	&=\frac {1}{s\phi(s)}  \int_0^{e^{-1}\phi(s)} (v\ln v)e^{-v}\,dv
		-\frac {\ln\phi(s)}{s\phi(s)}  \int_0^{e^{-1}\phi(s)} ve^{-v}\,dv\\
	&=: I_1(s)-I_2(s).
\end{align*}
Since $\phi(s)\in \mathcal{R}_{\beta}(\infty)$ with $\beta>0$, it follows from \cite[Proposition 1.3.6 (v)]{BGT} that $\lim_{s\to\infty}\phi(s)=\infty$, and hence,
$
	\lim_{s\to \infty} \frac{s\phi(s)}{\ln\phi(s)}I_2(s)=\int_0^\infty ve^{-v}\,dv=1.
$
On the other hand, since $\int_0^\infty v |\ln v|e^{-v}\,dv<\infty$, it follows that 
$
	\limsup_{s\to \infty} \bigl|\frac{s\phi(s)}{\ln\phi(s)}I_1(s)\bigr|
	\le \limsup_{s\to \infty} \frac{1}{\ln\phi(s)}\int_0^{e^{-1}\phi(s)} v|\ln v|e^{-v}\,dv=0.
$
Therefore, 
\beq\label{eqn:aux2}
	\mathcal{L}_t[g(t)](s) \sim -\frac{\ln\phi(s)}{s\phi(s)} 
	\ \textrm{as} \ s\to\infty.
\eeq

Now, note that 
$
\E[V(E_t)]=-g(t) +e^{-1}\P(E_{t}>e^{-1}).
$
Then by \eqref{eqn:aux1} and \eqref{eqn:aux2}, 
\[
\mathcal{L}_t[\E[V(E_t)]](s)\sim \frac{\ln\phi(s)}{s\phi(s)}=\frac{\ln\phi(s)}{\ell(s)}s^{-(1+\beta)} \ \textrm{as} \ s\to\infty,
\]
where we expressed $\phi(s)$ as $\phi(s)=s^\beta \ell(s)$ for some $\ell\in\mathcal{R}_0(\infty)$. 
By Karamata's Tauberian Theorem \cite[Theorem 1.7.1]{BGT},
the latter yields
$
\int_{0}^{t}\E[V(E_{u})]du \sim \frac{\ln\phi(t^{-1})}{\ell(t^{-1})}\frac{t^{1+\beta}}{\Gamma(2+\beta)}  
$
as $t\downarrow 0$.
Since $V$ is nondecreasing, so is the function $t\mapsto \E[V(E_t)]$, and the desired conclusion now follows from the monotone density theorem \cite[Theorem 1.7.2]{BGT}.\end{proof}

The above two lemmas allow us to establish the small-time asymptotic behavior of the function $\E[E_{t}\ln(1/E_{t})]$, which is needed to deal with the threshold case $\alpha=1$.

\begin{prop}\label{prop:asym log}
Let $E$ be the inverse of a subordinator $D$ with Laplace exponent $\phi\in\mathcal{R}_\beta(\infty)$ with $\beta\in(0,1)$. Then
\[
\E[E_{t}\ln(1/E_{t})]\sim  \frac{1}{\Gamma(1+\beta)}[\phi(1/t)]^{-1}\ln\phi(1/t) \text{ as } t\downarrow 0.
\]
\end{prop}

\begin{proof}
The definition of $V(x)$  in \eqref{eqn:F} yields
\[
	\E[E_{t}\ln(1/E_{t})]
	=\E[V(E_t)]-e^{-1}\P(E_{t}\geq e^{-1})+\E[E_{t}\ln(1/E_t); E_{t}\geq e^{-1}].
\]
This together with Lemma \ref{lemma:F} implies that the desired result follows upon showing that both $\P(E_{t}\geq e^{-1})$ and $\E[E_{t}\ln(1/E_t); E_{t}\geq e^{-1}]$ decay at a rate faster than $[\phi(1/t)]^{-1}\ln\phi(1/t)$ decays. However, since $\phi(1/t)^{-1}\ln\phi(1/t)$ has a polynomial decay due to the condition $\phi(s)\in\mathcal{R}_\beta(\infty)$, it suffices to prove that both $\P(E_{t}\geq e^{-1})$ and $\E[E_{t}\ln(1/E_t); E_{t}\geq e^{-1}]$ decay at least exponentially as $t\downarrow 0$. 
Now, \eqref{eqn:small1} immediately yields an exponential decay of $\P(E_{t}\geq e^{-1})$. 
On the other hand, in terms of $\E[E_{t}\ln(1/E_t); E_{t}\geq e^{-1}]$, choose a constant $c_{3}>0$ such that $|\ln x|\leq c_{3}x$ for all $x>e^{-1}$, so that
$
\left|\E[E_{t}\ln(E_{t}); E_{t}>e^{-1}]\right|\le c_3\E[E_t^{2}; E_{t}>e^{-1}]. 
$
Since the right-hand side decays exponentially due to \eqref{eqn:small2}, the left-hand side decays at least exponentially.
This completes the proof. 
\end{proof}

\begin{proof}[Proof of Theorem \ref{thm:SHC small time}]
By \cite[Proposition 4.2]{KP}, for any fixed $p>0$ and $\delta>0$,
\[
	\E[E_t^p]\sim \E[E_t^p; E_t\le \delta]\sim \frac{\Gamma(p+1)}{\Gamma(p\beta+1)}[\phi(1/t)]^{-p} \ \textrm{as} \ t\downarrow 0. 
\]
This yields the following statement in the case when $\alpha\in (0,1)\cup(1,2)$: 
\begin{align}\label{eqn:alpha_not_1}
\E[f_{\alpha}(E_{t}); E_{t}\leq \delta]\sim \E[f_{\alpha}(E_{t})]\sim
\begin{cases}
\frac{\Gamma(1+1/\alpha)}{\Gamma(1+\beta/\alpha)}[\phi(1/t)]^{-1/\alpha} &\text{if } \alpha\in (1,2),\\
\frac{1}{\Gamma(1+\beta)}[\phi(1/t)]^{-1} &\text{if } \alpha\in (0,1),
\end{cases}
\end{align}
where $f_\alpha(t)$ is defined in \eqref{def:f_alpha}.
On the other hand, 
application of \eqref{eqn:small2} and Proposition \ref{prop:asym log} yields 
\begin{align}\label{eqn:alpha_equal_1}
	\E[f_{\alpha}(E_{t}); E_{t}\leq \delta]\sim \E[f_{\alpha}(E_{t})]
	\sim \frac{1}{\Gamma(1+\beta)}[\phi(1/t)]^{-1}\ln\phi(1/t) \ \ \text{if } \alpha=1.
\end{align}
In particular, expressions \eqref{eqn:alpha_not_1} and \eqref{eqn:alpha_equal_1} show that regardless of the value of $\alpha\in(0,2)$,  
\beq\label{cond:main2}
	\E[f_{\alpha}(E_t); E_{t}\leq \delta_1] \sim \E[f_{\alpha}(E_t); E_{t}\leq \delta_2] \ \textrm{for any} \ \delta_1,\delta_2>0. 
\eeq 
Moreover, expressions \eqref{eqn:small1}, \eqref{eqn:alpha_not_1} and \eqref{eqn:alpha_equal_1} together imply that
\beq\label{cond:main1}
\P(E_{t}>\delta)=o(\E[f_{\alpha}(E_t); E_{t}\leq \delta]) \ \textrm{for any} \ \delta>0. 
\eeq
Applying the argument given in the proof of \cite[Proposition 2.2]{KP}, one can verify that Theorem \ref{thm:main} and the two statements \eqref{cond:main2} and \eqref{cond:main1} together imply
\[
|\Omega|-Q^{Y\circ E}_{\Omega}(t)\sim c_{\alpha}\E[f_{\alpha}(E_t); E_{t}\leq \delta]
\]
for any $\delta>0$, where $c_\alpha$ is defined in \eqref{def:f_alpha}. 
This is equivalent to the conclusion of Theorem \ref{thm:SHC small time}.
\end{proof}

\begin{remark}
\begin{em}
Suppose for simplicity that the time change $E$ in Theorem \ref{thm:SHC small time} is given by the inverse of a \textit{stable} subordinator with index $\beta\in(0,1)$. 
Then, the rate function for the decay of $|\Omega|-Q^{Y\circ E}_{\Omega}(t)$ as $t\downarrow 0$ is given by 
\[
	\begin{cases}
		t^{\beta/\alpha} \ &\textrm{if} \ \alpha\in(1,2),\\
		t^{\beta} \ln (1/t) \ &\textrm{if} \ \alpha=1,\\
		t^{\beta} \ &\textrm{if} \ \alpha\in(0,1).
	\end{cases}
\]
Comparing this with the statement of Theorem \ref{thm:main}, one can observe that the short-time decay rate for $|\Omega|-Q^{Y\circ E}_{\Omega}(t)$ is faster than that for $|\Omega|-Q^{Y}_{\Omega}(t)$, regardless of the values of the indices $\alpha\in(0,2)$ and $\beta\in(0,1)$. This makes sense since, even though the introduction of the inverse stable subordinator $E$ makes the heat particles diffuse at a slower rate in large time, they actually diffuse at a faster rate near $t=0$, and thus, more heat particles exit the domain $\Omega$ in short time.
\end{em}
\end{remark}

\begin{singlespace}

\end{singlespace}
\end{doublespace}

\vskip 0.3truein

{\bf Kei Kobayashi}

Department of Mathematics, Fordham University, NY 10023, USA

E-mail: \texttt{kkobayashi5@fordham.edu}

\vskip 0.3truein

{\bf Hyunchul Park}

Department of Mathematics, State University of New York at New Paltz, NY 12561,
USA

E-mail: \texttt{parkh@newpaltz.edu}


\begin{thebibliography}{99}


\bibitem{Val2016} 
L.~Acu\~{n}a Valverde.
On the one dimensional spectral heat content for stable processes.
\textit{J. Math. Anal. Appl.},  {\bf 441}  (2016),  11--24. 

\bibitem{Val2017} 
L.~Acu\~{n}a Valverde.
Heat content for stable processes in domains of $\R^{d}$.
\textit{J. Geom. Anal.}, {\bf 27} (2017),  492--524.

\bibitem{BGT}
N.H. Bingham, C.M. Goldie, and J.L. Teugels.
\newblock {\em Regular variation}, volume~27 of {\em Encyclopedia of
  Mathematics and its Applications}.
\newblock Cambridge University Press, 1987.


\bibitem{GPS19} T. Grzywny, H. Park, and R. Song.
Spectral heat content for L\'evy processes.
\textit{Math. Nachr.} 
\textbf{292} (2019), 805--825.


\bibitem{Kei} K. Kobayashi. 
Small ball probabilities for a class of time-changed self-similar processes. 
\textit{Statist. Probab. Lett.} \textbf{110} (2016), 155--161.

\bibitem{KP} K. Kobayashi and H. Park. 
Spectral heat content for time-changed killed Brownian motions. Submitted.
\href{https://arxiv.org/abs/2007.05776}{arXiv:2007.05776}.

\bibitem{MeerschaertSikorskii}
M.M. Meerschaert and A. Sikorskii.
\textit{Stochastic Models for Fractional Calculus.}
Volume 43 of De Gruyter Studies in Mathematics, 2012. 

\bibitem{P20} H. Park.
Higher order terms of spectral heat content for killed subordinate and subordinate killed Brownian motions related to symmetric $\alpha$-stable processes in $\R$.
To appear in \textit{Potential Analysis}. 

\bibitem{PS19} H. Park and R. Song. 
Small time asymptotics of spectral heat contents for subordinate killed Brownian motions related to isotropic $\alpha$-stable processes.
\textit{Bull. London Math. Soc.} 
\textbf{51} (2019), 371--384.

\bibitem{PS22} H. Park and R. Song.
Spectral heat content for $\alpha$-stable processes in $C^{1,1}$ open sets. 
\textit{Elect. J. Probab.}, \textbf{Vol. 27} No. 22,  (2022) 1--19.


\bibitem{Sato} K.\ Sato.
\textit{L{\'e}vy Processes and Infinitely Divisible Distributions.} 
Cambridge University Press, 1999.


\bibitem{HKU-book} S. Umarov, M. Hahn, and K. Kobayashi.
\textit{Beyond the Triangle: Brownian Motion, It\^o {C}alculus, and Fokker--Planck Equation --- Fractional Generalizations.}
World Scientific, 2018. 


\end{thebibliography}
\end{document}